\documentclass[11pt, a4paper]{article}
\usepackage{times}
\usepackage{a4wide}
\usepackage[dvips, hyperindex]{hyperref}
\usepackage[british]{babel}
\usepackage{enumerate, longtable}
\usepackage{amsmath, amscd, amsfonts, amsthm, amssymb, latexsym, comment, stmaryrd, graphicx}
\usepackage[all]{xy}
\usepackage[T1]{fontenc}
\usepackage[latin1]{inputenc}
\usepackage{amsfonts}
\usepackage{amssymb}

\newtheorem{thm}{Theorem}[section]
\newtheorem{lem}[thm]{Lemma}
\newtheorem{defi}[thm]{Definition}
\newtheorem{rem}[thm]{Remark}

\newtheorem{prop}[thm]{Proposition}

\newcommand{\GL}{\mathrm{GL}}
\newcommand{\GSp}{\mathrm{GSp}}
\newcommand{\PGSp}{\mathrm{PGSp}}
\newcommand{\PSp}{\mathrm{PSp}}
\newcommand{\Sp}{\mathrm{Sp}}
\newcommand{\SL}{\mathrm{SL}}
\newcommand{\SO}{\mathrm{SO}}
\newcommand{\GO}{\mathrm{GO}}

\newcommand{\HT}{\mathrm{HT}}
\newcommand{\WD}{\mathrm{WD}}
\newcommand{\W}{\mathrm{W}}
\newcommand{\unit}{\mathrm{unit}}

\DeclareMathOperator{\Det}{det}

\newcommand{\Ind}{{\rm Ind}}

\newcommand{\calL}{\mathcal{L}}

\newcommand{\cO}{\mathcal{O}}

\newcommand{\CC}{\mathbb{C}}
\newcommand{\FF}{\mathbb{F}}

\newcommand{\tr}{\mathrm{tr}}

\newcommand{\rec}{\mathrm{rec}}

\newcommand{\Q}{\mathbb{Q}}
\newcommand{\R}{\mathbb{R}}
\newcommand{\A}{{\mathbb{A}}}
\newcommand{\Z}{{\mathbb{Z}}}
\newcommand{\C}{{\mathbb{C}}}
\newcommand{\F}{{\mathbb{F}}}

\newcommand{\pl}{\hat{\mu}^{\mathrm{pl}}}
\def\ra{\rightarrow}
\def\lg{\langle}
\def\rg{\rangle}
\def\hra{\hookrightarrow}
\newcommand{\cL}{{\mathcal{L}}}
\newcommand{\ur}{{\rm ur}}
\newcommand{\triv}{\mathbf{1}}
\newcommand{\cyclo}{\chi_{\mathrm{cyc}}} 

\def\benu{\begin{enumerate}}
\def\eenu{\end{enumerate}}

\def\beq{\begin{equation}}
\def\eeq{\end{equation}}

\def\bit{\begin{itemize}}
\def\eit{\end{itemize}}

\begin{document}

\title{Compatible systems of symplectic Galois representations and the inverse
Galois problem III.\\
Automorphic construction of compatible systems with suitable local
properties.}
\author{
Sara Arias-de-Reyna\footnote{Universit\'{e} du Luxembourg,
Facult\'{e} des Sciences, de la Technologie et de la Communication,
6, rue Richard Coudenhove-Kalergi, L-1359 Luxembourg, Luxembourg,
sara.ariasdereyna@uni.lu}, Luis V. Dieulefait\footnote{Departament
d'Algebra i Geometria, Facultat de Matem\`{a}tiques, Universitat de
Barcelona, Gran Via de les Corts Catalanes, 585, 08007 Barcelona,
Spain, ldieulefait@ub.edu}, Sug Woo Shin\footnote{MIT, Department of Mathematics, 77 Massachusetts Avenue, Cambridge, MA 02139, USA / Korea Institute for Advanced Study, 85 Hoegiro,
Dongdaemun-gu, Seoul 130-722, Republic of Korea, swshin@mit.edu}, Gabor
Wiese\footnote{Universit\'{e} du Luxembourg, Facult\'{e} des
Sciences, de la Technologie et de la Communication, 6, rue Richard
Coudenhove-Kalergi, L-1359 Luxembourg, Luxembourg,
gabor.wiese@uni.lu}}
\maketitle

\begin{abstract}
This article is the third and last part of a series of three articles about
compatible systems of symplectic Galois representations and applications to the inverse
Galois problem.

This part proves the following new result for the inverse Galois
problem for symplectic groups. For any even positive integer~$n$ and any positive
integer~$d$, $\PSp_n(\FF_{\ell^d})$ or $\PGSp_n(\FF_{\ell^d})$ occurs
as a Galois group over the rational numbers for a positive density set of primes~$\ell$.

The result is obtained by showing the existence of a regular, algebraic,
self-dual, cuspidal automorphic representation of $\GL_n(\A_\Q)$
with local types chosen so as to obtain a compatible system of Galois representations
to which the results from Part~II of this series apply.


\end{abstract}

\section{Introduction}\label{sec:introduction}

This article is the last part of a series of three on compatible
systems of symplectic Galois representations and applications to the
inverse Galois problem (cf.\ \cite{partI}, \cite{partII}).
Our main theorem is the following new result for the inverse Galois
problem over $\Q$ for symplectic groups.

\begin{thm}\label{teo:InverseGalois}
For any even positive integer $n$ and for any positive integer $d$
there exists a set of rational primes of positive density such that, for
every prime $\ell$ in this set, the group $\PGSp_n(\F_{\ell^d})$ or
$\PSp_n(\F_{\ell^d})$ is realised as a Galois group over~$\Q$.
The corresponding number field ramifies at most at~$\ell$ and two more primes,
which are independent of~$\ell$.
\end{thm}

In fact, one of the two auxiliary primes can be taken to be any prime,
the other one can be chosen from a set of primes of positive density defined by
an explicit Chebotarev condition.
The set of primes~$\ell$ depends on the two previous choices and the choice
of an automorphic form, and is also given by a Chebotarev condition
(in the projective field of definition (see~\cite{partI}) of the compatible system
of Galois representations attached to the automorphic form (see below))
except for a density-zero set.

Theorem~\ref{teo:InverseGalois} is complementary to the main result of Khare, Larsen
and Savin~\cite{KLS1}, in the sense that it is in the horizontal
direction in the terminology of~\cite{DiWi}, whereas loc.~cit.\
is in the vertical one, that is, $\ell$ is fixed and $d$ runs.
The horizontal direction needs quite a different approach from the
vertical one. Nevertheless, some ideas of~\cite{KLS1}, for instance
that of $(n,p)$-groups, are crucially used also in our approach.
The overall strategy is described in the introduction to Part~I.

The goal in this Part~III is to construct compatible systems of Galois
representations satisfying the conditions in the main theorem on the inverse Galois problem
of~\cite{partII}.
In order to do so, we prove the existence of a regular, algebraic,
self-dual, cuspidal automorphic representation of $\GL_n(\A_\Q)$
with the required local types by adapting the results from~\cite{S}. This
automorphic representation is such that the compatible system of
Galois representations attached to it is symplectic, generically
irreducible, has a maximally induced place~$q$ of a certain prime order~$p$,
and locally at a prime $t$ contains a transvection. In order to show that the
transvection is preserved in the image of the residual
representation, at least for a density one set of primes, we will
apply a level-lowering result from~\cite{BLGGT} over suitable quadratic
imaginary fields.

The structure of this paper is the following. In Section~\ref{sec:RAESDC}
we recall general facts about regular self-dual
automorphic representations and their corresponding compatible
systems of Galois representations (everything in this section can be
found in~\cite{BLGGT}). In Section~\ref{sec:existence} we show the existence of the sought for
automorphic representation. In Section~\ref{sec:compatible-systems} we specify the conditions
on the ramified primes that we will need and explain the properties
of the compatible system attached to the automorphic representation from
Section~\ref{sec:existence}. In Section~\ref{sec:transvections} we perform the level-lowering argument.
Finally, in Section~\ref{sec:conclusion} we derive the main conclusions that follow
from the combination of the results in our three papers.

\subsection*{Acknowledgements}

S.~A.-d.-R. was partially supported by the project MTM2012-33830 of the Ministerio de Econom\'ia y Competitividad of Spain. She thanks the University of Barcelona for its hospitality during several short visits. S.~A.-d.-R. would like to thank X.~Caruso for his explanations concerning tame inertia weights and the Hodge-Tate weights. L. V. D. was supported by the project MTM2012-33830 of the Ministerio de Econom\'ia y Competitividad of Spain and by an ICREA Academia Research Prize.
S.~W.~S. was partially supported by NSF grant DMS-1162250 and Sloan Fellowship.
G.~W.\ was partially supported by the DFG priority program~1489 and by the Universit\'e du Luxembourg.

\section{RAESDC automorphic representations and the Galois
representations attached to them}\label{sec:RAESDC}

The reader is referred to \cite{BLGGT} for more details concerning
anything in this section except the $v=l$ case of \eqref{item:4} below, for which we refer to \cite{Caraiani}. For a field $k$ we adopt the notation $G_k$ to denote the absolute Galois group of $k$.

Let $\Z^{n, +}$ be the set of $n$-tuples $a = (a_i) \in \Z^n$ such
that $a_1 \geq a_2 \geq .... \geq a_n$.
Let $a  \in  \Z^{n, +} $, and let $\Xi_a$ be the irreducible
algebraic representation of $\GL_n$ with highest weight $a$.
A \emph{RAESDC (regular, algebraic, essentially self-dual, cuspidal)
automorphic representation of $\GL_n(\A_\Q)$} is a pair $(\pi, \mu)$
consisting of a cuspidal automorphic representation $\pi$ of $\GL_n(\A_\Q)$ and a continuous character $\mu: \A_\Q^{\times}/ \Q^{\times} \rightarrow \C^{\times}$ such that:

\begin{enumerate}[(1)]

\item (regular algebraic) $\pi_\infty$ has the same infinitesimal character as
$\Xi_a^{\vee}$ for $a \in \Z^{n, +}$. We say that $\pi$ has weight $a$.

\item  (essentially self-dual) $\pi \cong \pi^{\vee} \otimes (\mu \circ \Det)$.

\end{enumerate}

Such a pair $(\pi,\mu)$ is an instance of a polarized representation in the sense of \cite[2.1]{BLGGT}.
In this situation, there exists an integer $w$ such
that, for every $ 1 \leq i \leq n$, $a_i + a_{n+1-i} = w$.
Let $S$ be the (finite) set of primes $p$ such that $\pi_p$ is ramified.
There exist a number field $M\subset \C$, which is finite over the field of rationality of $\pi$ in the sense of \cite{Clo90}, and a strictly
compatible system of semisimple Galois representations (see \cite[5.1]{BLGGT} for this notion; in particular the characteristic polynomial of a Frobenius element at almost every finite place has coefficients in $M$)
\begin{equation*}\begin{aligned} \rho_\lambda (\pi): G_\Q &\rightarrow \GL_n (\overline{M}_\lambda),\\
\rho_\lambda(\mu): G_\Q &\rightarrow \overline{M}_\lambda^{\times}, \end{aligned}\end{equation*}
where $\lambda$ ranges over all finite places of~$M$
(together with fixed embeddings $M \hookrightarrow M_\lambda \hookrightarrow \overline{M}_\lambda$,
where $\overline{M}_\lambda$ is an algebraic closure of~$M_\lambda$)
such that the following properties are satisfied.

\begin{enumerate}[(1)]

\item  $\rho_\lambda(\pi) \cong \rho_\lambda(\pi)^{\vee} \otimes
\cyclo^{1-n} \rho_\lambda(\mu)$, where $\cyclo$ denotes the
$\ell$-adic cyclotomic character.

\item The representations $\rho_\lambda(\pi)$ and
$\rho_\lambda(\mu)$ are unramified outside $S \cup \{ \ell  \}$, where $\ell$ denotes the rational prime below $\lambda$.

\item  Locally at $\ell$ (i.e., when restricted to a decomposition group
at $\ell$), the representations $\rho_\lambda(\pi)$ and
$\rho_\lambda(\mu)$ are de Rham, and if $\ell \notin S$, they are
crystalline.

\item The set $\HT(\rho_\lambda(\pi))$ of Hodge-Tate weights of $\rho_\lambda(\pi)$
is equal to:
\begin{equation*} \{  a_1 + (n-1), a_2 + (n-2), \ldots, a_n  \}.\end{equation*}
(The Hodge-Tate weight of $\cyclo$ is $-1$.) In particular, they are $n$ different numbers and they are independent of $\lambda$ and $\ell$.
Therefore, the representations are regular.

\item\label{item:4}  The system is strictly compatible, as implied by the following
compatibility with Local Langlands: Fix any isomorphism $\iota:\overline{M}_\lambda\simeq \C$ compatible with the inclusion $M\subset \C$ w.r.t.\ the already fixed embedding $M \hookrightarrow M_\lambda \hookrightarrow \overline{M}_\lambda$.
Whether $v \nmid \ell$ or $v|\ell$, we have:
\begin{equation}\label{eq:star} \iota\WD(\rho_{\lambda}(\pi)|_{G_{\Q_v}})^{\mathrm{F-ss}} \cong \rec (\pi_v \otimes | \Det  |_v^{(1-n)/2}).\end{equation}
Here $\WD$ denotes the Weil-Deligne representation attached to a
representation of $G_{\Q_v}$,
$\mathrm{F-ss}$ means the Frobenius semisimplification, and $\rec$ is the notation for the (unitarily normalised) Local Langlands Correspondence, which attaches to an irreducible admissible representation of $\GL_n(\Q_v)$ a $\WD$ representation of the Weil group $\W_{\Q_v}$.

\end{enumerate}

\begin{rem}
We did not include $\iota$ in the notation since the isomorphism class of $\rho_{\lambda}(\pi)$ is independent of the choice of $\iota$. This is easy to deduce from the fact that the Frobenius traces at all but finitely many places are in $M$ via the Chebotarev density theorem.
\end{rem}

\begin{rem} For every prime $p$ of good reduction for~$\pi$, and for each $\lambda\nmid p$ prime of $M$, the trace of the image under $\rho_{\lambda}(\pi)$ of the Frobenius at $p$ belongs to the field of rationality of $\pi$ (hence to $M$) since the map $\rec$ twisted by $|\det|^{(1-n)/2}$ commutes with all field automorphisms of $\C$. Therefore, if the residual representation $\overline{\rho}_{\lambda}(\pi)$ is absolutely irreducible, it
follows from Th\'eor\`eme~2 of~\cite{Ca} that the representation $\rho_\lambda (\pi)$ can be
defined over $M_\lambda$.
\end{rem}

\begin{rem} Observe that the above property (\ref{item:4}) implies that as long as $v$ and
$\ell$ are different, the behaviour locally at $v$ of the
representations $\rho_\lambda(\pi)$ is independent of $\ell$, and it
can be determined (up to Frobenius semisimplification) from the admissible representation $\pi_v$, via Local Langlands.\end{rem}

Fix a symmetric form on $\overline{M}_\lambda^n$, and a symplectic form if $n$ is even. Thus we have the subgroup $\GO_{n}(\overline{M}_\lambda)$, and also $\GSp_{n}(\overline{M}_\lambda)$ if $n$ is even, of $\GL_{n}(\overline{M}_\lambda)$.
It is important for us to know a criterion for the image of $\rho_\lambda(\pi)$ to be contained in $\GSp_{n}(\overline{M}_\lambda)$ (for any fixed choice of symplectic form on $\overline{M}_\lambda^n$) up to conjugation. This is deduced from a result of Bella\"iche and Chenevier. Noting that $\mu$ must be an algebraic Hecke character, let $r\in \Z$ denote the unique integer such that $\mu|\cdot|^{-r}$ is a finite character. The integer ring in $\overline{M}_\lambda$ will be denoted $\cO_{\overline{M}_{\lambda}}$.

\begin{lem}\label{l:image-GSp}
Suppose that $\rho_\lambda(\pi)$ is residually irreducible. If $n$ is even and $\mu_\infty(-1)=(-1)^r$ (resp. if $n$ is odd or $\mu_\infty(-1)\neq(-1)^r$) then the image of $\rho_\lambda(\pi)$ is contained in $\GSp_{n}(\cO_{\overline{M}_{\lambda}})$ (resp. $\GO_{n}(\cO_{\overline{M}_{\lambda}})$) possibly after a conjugation by an element of $\GL_n(\overline{M}_\lambda)$.
\end{lem}

\begin{rem}
When $n$ is even, it may happen that the image of $\rho_\lambda(\pi)$ is contained in $\GSp_{n}(\cO_{\overline{M}_{\lambda}})$ as well as $\GO_{n}(\cO_{\overline{M}_{\lambda}})$ after conjugation, cf.\ footnote~1 of \cite{BC11} for an example when $n=2$. The point is that while $\rho_\lambda(\pi)$ is completely characterised by $\pi$ alone, there are generally several choices of $\mu$ for the same $\pi$ with different values of $\mu_\infty(-1)$.
\end{rem}

\begin{proof}
By an argument as in Remark 2.2 of \cite{partI} (where the residual irreducibility is used), the proof is reduced to showing that the image is contained in either $\GSp_{n}(\overline{M}_{\lambda})$ or $\GO_{n}(\overline{M}_{\lambda})$.

We start with an easy observation. Let $\rho:\Gamma\ra \GL(V)$ be an irreducible representation on an $n$-dimensional vector space $V$ over $\overline{M}_\lambda$. Fix a basis $\{e_i\}$ for $V$ and write $\{e_i^\vee\}$ for the dual basis. Suppose that $(\rho^\vee,V^\vee)\simeq (\rho\otimes \chi,V)$ for a character $\chi$ of $\Gamma$. Let $A$ be an $n\times n$ matrix representing one such isomorphism. Then $A^{t}=\delta A$ for $\delta\in \{\pm1\}$, which is called the sign of $(\rho,\chi)$. (See the introduction of \cite{BC11}.) Then elementary linear algebra shows that if $\delta=1$ (resp. $\delta=-1$) then there exists a nondegenerate symmetric (resp. alternating) form $V\otimes V\ra \overline{M}_\lambda$ such that
$B(\gamma v,\gamma w)=\chi(\gamma)B(v,w)$ for $\gamma\in \Gamma$ and $v,w\in V$.

So the lemma amounts to the assertion that $(\rho,\chi)=(\rho_\lambda(\pi),\rho_\lambda(\mu))$ has sign 1 (resp. $-1$) if $n$ is even and $\mu_\infty(-1)=(-1)^r$ (resp.\ otherwise). This is exactly \cite[Cor 1.3]{BC11}.
\end{proof}

\section{Existence of self-dual automorphic representations with prescribed local conditions}\label{sec:existence}

The goal of this section is to prove the existence of a regular algebraic self-dual
cuspidal automorphic representation of $\GL_n(\A_\Q)$ with some particular local properties
when $n$ is even, for the application to the inverse Galois problem in Section~\ref{sec:conclusion}.
As we utilise Arthur's classification for representations of classical groups, our result depends on a
few hypotheses which his work depends on. (See Remark~\ref{r:hypotheses} below.) The reader may skip to
the next section after getting familiar with the notation of \S\ref{sub:app-to-existence} if he or she is
willing to accept Theorem~\ref{t:existence-self-dual} below.

We adopt the convention that all irreducible representations of $p$-adic or real groups are assumed admissible.
Whenever it is clear from the context, we often write a representation or an $L$-parameter to mean an appropriate
isomorphism or equivalence class thereof in favour of simplicity. We did not specialize to $F=\Q$ when recalling facts in
\S\ref{sub:existence} and \S\ref{sub:Arthur} below since the exposition hardly simplifies by doing so. The attentive reader
will notice that Theorem \ref{t:existence-self-dual} easily extends to the case over any totally real field.

\subsection{Plancherel measures}\label{sub:Plancherel}

This subsection is a reminder of some facts about Plancherel measures on $p$-adic groups. Let $G$ be a connected
reductive group over a non-archimedean or archimedean local field $K$. Write $G(K)^{\wedge}$ for the unitary dual
of $G(K)$, namely the set of all irreducible unitary representations of $G(K)$ equipped with the Fell topology.
Let $X^\ur_\unit(G(K))$ denote the set of all unitary unramified characters of $G(K)$ in the usual sense.
(This is $\mathrm{Im} X(G)$ on page 239 of \cite{Wal03}.)  Harish-Chandra proved
(cf.~\cite{Wal03}) that there is a natural Borel measure $\pl$ on $G(K)^{\wedge}$, called the Plancherel measure, satisfying
\begin{equation*}\phi(1)=\int_{G(K)^{\wedge}} \hat{\phi}(\tau) \pl,\quad \phi\in C^\infty_c(G(K)),\end{equation*}
where $\hat{\phi}$ is the function defined by $\hat{\phi}(\tau):=\tr \tau(\phi)$. Let $\Theta(G(K))$ denote the Bernstein variety,
which is a (typically infinite) disjoint union of affine complex algebraic varieties over $\C$. Viewing $\Theta(G(K))$ also as the
topological space on $\C$-points, the association of supercuspidal support defines a continuous map $\nu:G(K)^{\wedge}\ra \Theta(G(K))$
(\cite[Thm 2.2]{Tad88}).
In prescribing local conditions for automorphic representations we will often consider the following kind of subsets. The terminology
is non-standard and only introduced to save words.

\begin{defi}\label{d:prescribable}
  A subset $\hat{U}$ of the unitary dual $G(K)^{\wedge}$ is said to be \emph{prescribable} if
  \begin{itemize}
    \item $\hat{U}$ is a Borel set which is $\pl$-measurable with finite positive volume,
    \item $\nu(\hat{U})$ is contained in a compact subset of $\Theta(G(K))$, and
    \item for each Levi subgroup $L$ of $G$ and each discrete series $\sigma$ of $L(K)$, consider the function on
    $X^\ur_\unit(L(K))$ whose value at $\chi$ is the number of irreducible subquotients lying in $\hat{U}$ (counted with multiplicity)
    of the normalized induction of $\sigma\otimes \chi$ to $G$. Then the set of points of discontinuity has measure 0.
  \end{itemize}
\end{defi}

To show the flavour of this somewhat technical definition, we mention three examples for such subsets. The first example is
 the subset of unramified (resp. spherical) representations in
$G(K)^{\wedge}$ if $K$ is non-archimedean (resp. archimedean). The second example is the set of all $\tau\in G(K)^{\wedge}$ in a fixed Bernstein component, i.e.\ those
$\tau$ with the same supercuspidal support up to inertia equivalence. Finally the set $\{\tau\otimes \chi:\chi\in X^\ur_\unit(G(K))\}$
for a unitary discrete series $\tau$ of $G(K)$ also satisfies the requirements. (By a discrete series we mean an irreducible representation
whose matrix coefficients are square-integrable modulo center.) If $G$ is anisotropic over $K$ (which is true if $G$ is semisimple) then
$X^\ur_\unit(G(K))$ is trivial so the last example is a singleton.

\subsection{Existence of automorphic representations}\label{sub:existence}

In this subsection we recall one of the few existence theorems in~\cite{S}, which are based on the principle that the local
components of automorphic representations at a fixed prime are equidistributed in the unitary dual according to the Plancherel
measure. The reader is invited to see its introduction for more references in this direction.
There is a different approach to the existence of cuspidal automorphic representations via Poincar\'{e} series
(without proving equidistribution), cf.\ \cite[\S4]{KLS1}, \cite{Mui10}.

Let $G$ be a connected reductive group over a totally real number field $F$ such that
\begin{itemize}
  \item $G$ has trivial center and
  \item $G(F_w)$ contains an $\R$-elliptic maximal torus for every real place $w$ of $F$.
\end{itemize}
The first condition was assumed in~\cite{S} and it is kept here as it is harmless for our purpose below. However it
should be possible to dispense with the condition by fixing central character in the trace formula argument there.
Now let $S$ be a finite set of finite places of $F$. Let $\pl_v$ denote the Plancherel measure on $G(F_v)^\wedge$ for
$v\in S$. Let $\hat{U}_v\subset G(F_v)^{\wedge}$ be a prescribable subset for each $v\in S$.

\begin{prop}\label{p:existence2} There exists a cuspidal automorphic representation $\tau$ of $G(\A_F)$ such that
\begin{enumerate}
\item $\tau_v\in \hat{U}_v$ for all $v\in S$,
\item $\tau$ is unramified at all finite places away from $S$,
\item $\tau_w$ is a discrete series whose infinitesimal character is sufficiently regular for every infinite place $w$.
\end{enumerate}
\end{prop}

The regularity condition above should be explained. Fix a maximal torus $T$ and a Borel subgroup $B$ containing $T$ in
$G$ over $\C$ (the base change of $G$ to $\C$ via $w:F\hra \C$). Let $\Omega$ denote the Weyl group of $T$ in $G$.
The infinitesimal character $\chi_w$ of $\tau_w$ above, which may be viewed as an element of $X^*(T)\otimes_\Z \Q$, is
\emph{sufficiently regular} if there is an $\sigma\in \Omega$ such that $\lg \sigma\chi_w,\alpha^\vee\rg\gg0$ for every
$B$-positive coroot $\alpha^\vee$ of $T$ in $G$. (Precisely the condition is that $\lg \sigma\chi_w,\alpha^\vee\rg\ge C$, where $C$
is a large enough constant depending only on $G$, $S$ and $\{\hat U_v\}_{v\in S}$.) This condition is independent of the choice of $T$ and $B$.

\begin{proof}
 Our proposition is the analogue of Theorem~5.8 of~\cite{S} except that a priori a weaker condition on
 $\hat{U}:=\prod_{v\in S} \hat{U}_v$ is assumed here. Let us explain this point. By the very definition of
 prescribable subsets, the characteristic function on $\hat{U}$, denoted $\triv_{\hat{U}}$, belongs to the
 class of functions to which Sauvageot's density principle \cite[Thm 7.3]{Sau97} applies. (Our last condition
 in Definition \ref{d:prescribable} corresponds to condition (a)(1) in his theorem.) Hence we can take
 $\hat{f}_S=\triv_{\hat{U}}$ in \cite[Thm 4.11]{S} so that we have the analogue of \cite[Cor 4.12]{S} for our
 $\hat{U}$. Therefore the analogue of \cite[Thm 5.8]{S} is deduced through the same argument deriving that theorem
 from \cite[Cor 4.12]{S} originally.
\end{proof}

\subsection{Arthur's endoscopic classification for $\SO_{2m+1}$}\label{sub:Arthur}

Arthur \cite{Arthur} classified local and global automorphic representations of symplectic and special orthogonal
groups via twisted endoscopy relative to general linear groups. For our purpose it suffices to recall some facts
in the case of odd orthogonal groups.

Let $F$ be a number field. Denote by $\SO_{2m+1}$ the split special orthogonal group over $F$. Note that the dual
group of $\SO_{2m+1}$ is $\Sp_{2m}(\C)$. Write $$\xi:\Sp_{2m}(\C)\hra \GL_{2m}(\C)$$ for the standard embedding,
and put $\cL_{F_v}:=\W_{F_v}$ (resp. $\cL_{F_v}:=\W_{F_v}\times \SL_2(\C)$) according as $v$ is an infinite (resp. finite)
place. An $L$-parameter $\phi_v:\cL_{F_v}\ra \GL_{2m}(\C)$ is said to be \emph{of symplectic type} if it preserves a suitable
symplectic form on the $2m$-dimensional space, or equivalently, if $\phi_v$ factors through $\xi$ (after conjugating by an
element of $\GL_{2m}(\C)$).
For a place $v$ of $F$, let $\rec_v$ denote the (unitarily normalised) local Langlands bijection from the set of irreducible
representations of $\GL_{r}(F_v)$ to the set of $L$-parameters $\cL_{F_v}\ra \GL_r(\C)$ for any $r\in \Z_{\ge 1}$. When $v$ is
finite, there is a standard dictionary for going between local $L$-parameters for $\GL_r$ and $r$-dimensional Frobenius-semisimple
Weil-Deligne representations of $\W_{F_v}$ in a bijective manner.

For each local $L$-parameter $\phi_v:\cL_{F_v}\ra \Sp_{2m}(\C)$ (or a local $L$-parameter for $\GL_{2m}$ of symplectic type), Arthur
associates an $L$-packet $\Pi_{\phi_v}$ consisting of finitely many irreducible representations of $\SO_{2m+1}(F_v)$. Moreover each
irreducible representation belongs to the $L$-packet for a unique parameter (up to equivalence). If $\phi_v$ has finite centraliser
group in $\Sp_{2m}(\C)$ so that it is a discrete parameter, then $\Pi_{\phi_v}$ consists only of discrete series. A similar construction
was known earlier by Langlands (deriving from Harish-Chandra's results on real reductive groups) when $v$ is an infinite place of $F$.

Now let $\tau$ be a discrete automorphic representation of $\SO_{2m+1}(\A_F)$. Arthur shows the existence of a self-dual isobaric
automorphic representation $\pi$ of $\GL_{2m}(\A_F)$ which is a functorial lift of $\tau$ along the embedding $\Sp_{2m}(\C)\hra \GL_{2m}(\C)$.
In the generic case (in Arthur's sense, i.e.\ when the $\SL_2$-factor in the global $A$-parameter for $\tau$ has trivial image),
this means that for the unique $\phi_v$ such that $\tau_v\in \Pi_{\phi_v}$, we have
\begin{equation*}\rec_v(\pi_v)\simeq \xi\phi_v,\quad \forall v.\end{equation*}

\begin{rem}\label{r:hypotheses}
Arthur's result \cite{Arthur} is conditional on the stabilisation of the twisted trace formula and a few expected technical
results in harmonic analysis as explained there.
See \cite[1.18]{BMM} for a summary of these issues. We also note that the references [A24]-[A28] in Arthur's book are
in preparation at the time our paper is finished. Since these results are expected to become available from
ongoing projects by others or by Arthur himself, the authors think that one need not strive hard to avoid using them. However see
Remark~\ref{r:alternative} below for a possible alternative approach.
\end{rem}

\subsection{Application: Existence of self-dual representations}\label{sub:app-to-existence}

We seek self-dual representations with specific local conditions. To describe them we need to set things up. It
is enough to restrict the material of previous subsections to the case $F=\Q$.
Let $n$ be a positive \emph{even} integer.
Let $p,q,t$ be distinct rational primes and assume that the order of $q$ mod $p$ is $n$. Denote by $\Q_{q^n}$ an
unramified extension of $\Q_q$ of degree $n$.
Choose a tamely ramified character $\chi_q:\Q_{q^n}^\times\ra \C^\times$ of order $2p$ such that $\chi_q(q)=-1$ and
$\chi_q|_{\Z_{q^n}^{\times}}$ is of order $p$ (cf.\ Section 3.1 of \cite{KLS1}
and the definition of maximal induced places of order~$p$ made in Part~I \cite{partI}). By local class field theory
we also regard $\chi_q$ as a character of $G_{\Q_q}$ (or $\W_{\Q_q}$). Put
 \begin{equation*}\rho_q:=\Ind^{G_{\Q_{q}}}_{G_{\Q_{q^{n}}}} (\chi_q).\end{equation*}
We write $\WD(\rho_q)$ for the associated Weil-Deligne representation, giving rise to a local $L$-parameter $\phi_q$ for
$\GL_n(\Q_q)$. Since $\rho_q$ is irreducible and symplectic (\cite[Prop 3.1]{KLS1}), the parameter $\phi_q$ factors through
$\widehat{\SO_{n+1}(\C)}=\Sp_{n}(\C)\subset \GL_n(\C)$ (after conjugation if necessary) and defines a discrete $L$-parameter
of $\SO_{n+1}(\Q_q)$. So the $L$-packet $\Pi_{\phi_q}$ consists of finitely many discrete series of $\SO_{n+1}(\Q_q)$.
(In fact \cite[\S5.3]{KLS1} exploits the fact that $\Pi_{\phi_q}$ contains a generic supercuspidal representation as shown by
Jiang and Soudry. In our method it suffices to have the weaker fact that $\pi_{\phi_q}$ contains a discrete series.)

It is a little more complicated to explain the objects at $t$. Let $\mathrm{St}_2$ denote the Steinberg representation of
$\GL_2(\Q_t)$ which appears as a subquotient of the unnormalized parabolic induction of the trivial character. Since
$\mathrm{St}_2$ has trivial central character, it corresponds to an $L$-parameter valued in $\SL_2(\C)$. Let $M$ be a Levi subgroup of
$\SO_{n+1}$ isomorphic to $\SO_3\times (\GL_1)^{\frac{n}{2}-1}$ so that $\hat{M}\simeq \SL_2(\C)\times \GL_1(\C)^{\frac{n}{2}-1}$.
Consider a local $L$-parameter $\phi^M_t:\W_{\Q_t}\times \SL_2(\C)\ra \hat{M}$ having the form
\begin{equation*}\phi^M_t=\left(\rec_t(\mathrm{St}_2),\phi_1,\cdots ,\phi_{\frac{n}{2}-1} \right),\quad \phi_i\in X^{\ur}_\unit(\W_{\Q_t}),\end{equation*}
where $X^{\ur}_\unit(\W_{\Q_t})$ denotes the set of unitary unramified characters of $\W_{\Q_t}$. Then $\phi^M_t$ is a discrete parameter for $M$.
Writing $\eta^M:\hat{M}\hra \Sp_n(\C)$ for a Levi embedding (canonical up to conjugacy), we see that $\xi\eta^M\phi^M_t$ is the $L$-parameter for
$\GL_n$ corresponding to the $n$-dimensional Weil-Deligne representation
\begin{equation*}\rec_t(\mathrm{St}_2)\oplus \left( \bigoplus_{i=1}^{\frac{n}{2}-1} \phi_i\oplus \phi_i^{-1}\right).\end{equation*}
Arthur associates a local $L$-packet $\Pi_{\phi^M_t}$ of irreducible $M(\Q_t)$-representations to $\phi^M_t$.
Define $\hat{U}_t$ to be the set of all irreducible subquotients of the parabolic induction of $\tau^M_t\otimes \chi$ from
$M(\Q_t)$ to $\SO_{n+1}(\Q_t)$ as $\tau^M_t$ runs over $\Pi_{\phi^M_t}$ and $\chi$ runs over $X^\ur_\unit(M(\Q_t))$. Since the
effect of the parabolic induction on $L$-parameters is simply composition with $\eta^M$ (this is implicit in the proof of the
proposition 2.4.3 in \cite{Arthur}), the $L$-parameter for any $\tau_t\in \hat{U}_t$ has the following form (after composing with $\xi$):
\begin{equation}\label{e:Galois-at-t}
   \lambda_0\rec_t(\mathrm{St}_2)\oplus \left( \bigoplus_{i=1}^{\frac{n}{2}-1} \lambda_i\phi_i\oplus \lambda_i^{-1}\phi_i^{-1}\right),
 \quad \lambda_i\in X^{\ur}_\unit(\W_{\Q_t}),~0\le i\le \frac{n}{2}-1.
\end{equation}
Notice that $\hat{U}_t$ is a prescribable subset. Among the conditions of Definition \ref{d:prescribable} we only check the last
one as the others are reasonably easy. By the way $\hat{U}_t$ is constructed, the function on $X^\ur_\unit(L(F))$ in that condition
is either identically zero unless $L=M$ (up to conjugacy) and $\sigma$ is in the $X^\ur_\unit(L(F))$-orbit of $\tau^M_t$ for some
$\tau^M_t\in \Pi_{\phi^M_t}$. So we may assume that $L=M$ and that $\sigma=\tau^M_t$. Then a version of the generic irreducibility
of parabolic induction (see \cite[Thm 3.2]{Sau97} attributed to Waldspurger, cf. the fourth entry in Appendix A of \cite{S} for a minor correction)
implies that the function is the constant function 1 away from a closed measure zero set. This verifies the condition as desired.

\begin{thm}\label{t:existence-self-dual}
There exists a cuspidal automorphic representation $\pi$ of $\GL_{n}(\A_\Q)$ such that
\begin{enumerate}[(i)]
\item\label{item:i} $\pi$ is unramified away from $q$ and $t$,
\item $ \rec_q(\pi_q)\simeq \WD(\rho_q)$,
\item $\rec_t(\pi_t)$ has the form \eqref{e:Galois-at-t},
\item $\pi_\infty$ is of symplectic type and regular algebraic,
\item\label{item:v} $\pi\simeq \pi^\vee$, and
\item\label{item:vi} the central character of $\pi$ is trivial.
\end{enumerate}
\end{thm}

As we rely on Arthur's work \cite{Arthur}, our theorem is conditional as explained in Remark~\ref{r:hypotheses}.

\begin{proof}
Apply Proposition~\ref{p:existence2} with $S=\{q,t\}$, $\hat{U}_q=\Pi_{\phi_q}$ and $\hat{U}_t$ as above. We have seen that
$\hat{U}_q$ and $\hat{U}_t$ are prescribable. Hence there exists a cuspidal automorphic representation $\tau$ of $\SO_{2n+1}(\A_\Q)$ such that
\begin{enumerate}
\item $\tau_v$ is unramified at every finite place $v\notin \{q,t\}$,
\item $\tau_q\in \Pi_{\phi_q}$,
\item $\tau_t\in \hat{U}_t$,
\item $\tau_\infty$ is a discrete series whose infinitesimal character is sufficiently regular.
\end{enumerate}
Then the functorial lift $\pi$ of $\tau$ as in \S\ref{sub:Arthur} has the desired properties (\ref{item:i})--(\ref{item:v})
of the theorem by construction. (To verify the regular algebraicity of $\pi$, it is enough to note that the $2n$-many exponents
for $z/\bar{z}$ in the parameter for $\pi_\infty$, coming from that of $\tau_\infty$, at infinite places are in $\frac{1}{2}\Z\backslash \Z$
and mutually distinct, cf.\ the bottom line of \cite[p.557]{KLS1}.) To see the cuspidality of $\pi$, it suffices to note that $\pi_q$ is
supercuspidal since $\rec_q(\pi_q)\simeq \WD(\rho_q)$, which is irreducible.
Finally (\ref{item:vi}) is derived from the fact that the central character is trivial at (almost) all finite places. Indeed the central
character corresponds  to the determinant of the $L$-parameter for $\pi$ at each place via local class field theory, but the determinant
is trivial since the parameter factors through $\SO_{n+1}(\C)$.
\end{proof}

\begin{rem}\label{r:alternative}
A slightly different version for the existence of $\tau$ can be shown by using Theorem~5.13 of~\cite{S}, which is obtained via the simple
trace formula for $G$, instead of using Proposition~\ref{p:existence2}. Then one can prescribe $\tau_\infty$ to be any favourite discrete series
at the expense of losing control of ramification at two auxiliary primes. On the other hand, it is conceivable that
one can prove the existence of $\pi$ directly, without going through representations of $G$ (thus avoiding the use of twisted endoscopy),
via the simple twisted trace formula for $\GL_n$ due to Deligne-Kazhdan as long as one is willing to allow ramification at one auxiliary
prime in (\ref{item:i}). (This is allowable for our purpose.) The idea would be to prescribe test functions at $q$, $t$ and infinite places
using the twisted Paley-Wiener theorem due to Rogawski and Mezo, cf.\ \cite[Thm 3.2]{CC09}.
\end{rem}

\section{Compatible systems attached to the RAESDC representations from the previous section}\label{sec:compatible-systems}

Let $n$ be an even integer and let $t$ be an arbitrary prime.
Choose a prime $p$ different from~$t$ such that $p \equiv 1 \mod n$.
Chebotarev's density theorem allows us to choose a prime $q$ different from~$t$
(out of a positive density set)
such that $q^{n/2} \equiv -1 \mod p$.

From now on, we will restrict to triples of primes satisfying these conditions, and we will keep the
notation $p, q, t$ for these primes, chosen in the order specified above. This notation is also compatible with the one in the previous section.

Consider an automorphic representation $\pi$ as in Theorem~\ref{t:existence-self-dual} and the
trivial character $\triv$ of $\A_{\Q}^\times/\Q^\times$. Then $(\pi, \triv)$ is a RAESDC representation
by construction, to which the facts of Section~\ref{sec:RAESDC} apply. So there are a number field $M \subset \CC$
and a compatible system of Galois representations $\rho_{\lambda}(\pi):G_{\Q}\ra \GL_n(\overline{M}_\lambda)$ as
$\lambda$ ranges over all places of $M$ with the following properties for every $\lambda$ (below, $\ell$ denotes the rational prime dividing $\lambda$).

\begin{enumerate}[(1)]
\item $\rho_\lambda(\pi) \cong \rho_\lambda(\pi)^\vee \otimes \cyclo^{1-n}$ and $\det \rho_{\lambda}(\pi)=\cyclo^{n(1-n)/2}$,
\item $\rho_\lambda(\pi)$ and its residual representation are absolutely irreducible if $\lambda\nmid p,q$,
\item $\rho_{\lambda}(\pi)$ is unramified away from $R=\{q,t\}$ and the residue characteristic of $\lambda$,
\item $\rho_{\lambda}(\pi)|_{G_{\Q_\ell}}$ is de Rham and has Hodge-Tate weights as described in Section~\ref{sec:RAESDC};
in particular they are independent of $\lambda$ and distinct,
\item $\rho_{\lambda}(\pi)|_{G_{\Q_\ell}}$ is crystalline if $\ell\notin R$,
\item $\rho_{\lambda}(\pi)$ has image in $\GSp_{n}(\cO_{\overline{M}_\lambda})$ (after a conjugation) by Lemma~\ref{l:image-GSp} if $\lambda \nmid p,q$,
\item The multiplier of $\rho_{\lambda}(\pi)$ is $\cyclo^{1-n}$,
\item\label{item:atq} $\WD(\rho_{\lambda}(\pi)|_{G_{\Q_q}})^{\mathrm{F-ss}}\simeq \WD(\rho_q)\otimes |\cdot|^{(1-n)/2}$ if
$\lambda\nmid q$. In particular, $\rho_{\lambda}(\pi)|_{G_{\Q_q}}\simeq \rho_q\otimes \alpha_{\lambda}$, where
$\alpha_{\lambda}:G_{\Q_q}\rightarrow \overline{M}_{\lambda}^{\times}$ is an unramified character (i.e.\ $q$ is a
maximally induced place of order~$p$ in the terminology of Part~I~\cite{partI}), and
\item\label{item:att} $\WD(\rho_{\lambda}(\pi)|_{G_{\Q_t}})^{\mathrm{F-ss}}$ has the form \eqref{e:Galois-at-t} if
$\lambda\nmid t$. In particular the image of $I_t$ under $\rho_\lambda(\pi)$ is
generated by a transvection, because $\rec_t(\mathrm{St}_2)$ is the Weil-Deligne representation having restriction to $I_t$ of the form
$(1, N)$ with $N =          \left(
              \begin{array}{cc}
                0 & 1  \\
                0 & 0
              \end{array}
            \right) $.

\end{enumerate}

All these are easy consequences of the facts recollected in Section~\ref{sec:RAESDC}. Note that for $\lambda\nmid p, q$,
the residual representation of $\rho_{\lambda}(\pi)$ is irreducible since its restriction to $G_{\Q_q}$ is irreducible
for the reason that $\rho_q$ modulo $\lambda$ is irreducible (\cite[3.1]{KLS1}). In particular the irreducibility
hypothesis of Lemma~\ref{l:image-GSp} is satisfied. After a conjugation we may and will assume that the image lies
in $\GSp_{n}(\cO_{M_\lambda})$. The determinant of $\rho_{\lambda}(\pi)$ is computed easily by relating it to the central
character of $\pi$, which is trivial (Theorem~\ref{t:existence-self-dual}~(\ref{item:vi})), keeping in mind that the
normalisation of the correspondence involves a twist by $|\cdot|^{(1-n)/2}$, cf.\ part~(\ref{item:4}) of Section~\ref{sec:RAESDC}.
Finally in order to see the second isomorphism follows from the first in (8), observe that $\WD(\rho_q)$ is already Frobenius semisimple.
Since $\rho_q$ is irreducible, this forces $\
WD(\rho_{\lambda}(\pi)|_{G_{\Q_q}})$ to be already Frobenius semisimple.
Thus $\WD(\rho_{\lambda}(\pi)|_{G_{\Q_q}})\simeq \WD(\rho_q)\otimes |\cdot|^{(1-n)/2}$, which implies that
$\rho_{\lambda}(\pi)|_{G_{\Q_q}}\simeq \rho_q\otimes |\cdot|^{(1-n)/2}$ since $\WD$ is a fully faithful functor.

\section{Level-lowering and transvections}\label{sec:transvections}

Throughout this section $\pi$ will denote the (fixed) automorphic representation considered in Section~\ref{sec:compatible-systems}, and $(\rho_{\lambda}(\pi))_\lambda$, for $\lambda$ running through the primes of a number field $M \subset \CC$, the compatible system attached to~$\pi$.
As before we denote by $\overline{\rho}_\lambda(\pi)$ the residual representation of~$\rho_\lambda(\pi)$.

The aim of this section is to study the transvection contained in
the image of $\rho_\lambda(\pi)$ given by its restriction to $I_t$,
for every $\ell \neq t$, $\lambda \mid \ell$. We want to show that
when reducing modulo $\lambda$ this transvection is preserved, or
equivalently, that the residual mod $\lambda$ representation is
ramified at $t$, at least for a density one set of primes $\ell$ and
every $\lambda \mid \ell$. The main tool will be a level-lowering argument based on Theorem~4.4.1 of~\cite{BLGGT}.

One of the hypotheses in this theorem is that, when restricted to the Galois group of a suitable cyclotomic extension, the residual representation is irreducible. The following lemma will be used to meet this requirement:

\begin{lem}\label{teo:strongirreducible} Let $F$ be a
quadratic number field. Then $\overline{\rho}_\lambda(\pi)|_{ G_{F(\zeta_\ell)}}$ is absolutely irreducible for almost every $\ell$ (i.e., for all but finitely many primes),
and for every $\lambda \mid \ell$. Moreover, the finite exceptional set can be bounded independently of~$F$.
\end{lem}

\begin{proof} Take $\ell$ to be a sufficiently large prime, such that in particular $\pi$ is not ramified at $\ell$ and $\ell$ is larger than the difference between any pair of Hodge-Tate weights plus $2$. Further we will assume $\ell\not=p, q$, so that we know that, because of the local parameter at $q$, $\overline{\rho}_{\lambda}(\pi)$ is absolutely irreducible and $\rho_{\lambda}(\pi)$ can be defined over $M_{\lambda}$. Call $\kappa_{\lambda}$ the residue field of $M_{\lambda}$.

Assume that $\ell>[M:\mathbb{Q}]$. Because of regularity (and constancy of Hodge-Tate weights), Proposition~5.1.2 and Lemma~1.1.4 of \cite{BL} (the latter is used to get rid of case (3) of the former) can be applied to conclude that either the residual mod $\lambda$ representation is reducible, or it is induced, or the group $\kappa_{\lambda}^{\times}\mathrm{im}(\overline{\rho}_{\lambda}(\pi))$ is a sturdy subgroup of $\GL_n(\kappa_{\lambda})$ (see \cite{BL} for the definition of sturdy subgroup). We already know that $\overline{\rho}_{\lambda}(\pi)$ is irreducible (because of the local parameter at $q$). Also, if we assume that it is induced, then we obtain a contradiction (for $\ell$ sufficiently large) by
Section~3 of \cite{partII}.
Thus, we conclude that there exists some bound $B$ such that,  for all $\ell>B$, $\kappa_{\lambda}^{\times}\mathrm{im}(\overline{\rho}_{\lambda}(\pi))$ is a sturdy subgroup of $\GL_n(\kappa_{\lambda})$. As a consequence of Lemma~4.2.2 of~\cite{BL}, $\kappa_{\lambda}^{\times}\mathrm{im}(\overline{\rho}_{\lambda}(\pi)\vert_{G_{F(\zeta_{\ell})}})$  is also a sturdy subgroup of $\GL_n(\kappa_{\lambda})$.
Thus, in particular, $\overline{\rho}_{\lambda}(\pi)\vert_{G_{F(\zeta_{\ell})}}$ is irreducible.
\end{proof}

From the lemma above, we know that there is a bound $B$ such that
for all primes $\ell > B$ and every $\lambda \mid \ell$ the
conclusion of the lemma holds for any quadratic field~$F$. In what
follows, we exclude the primes $\ell \leq B$.

The following lemma allows us to control the case where a residual
representation in our system can lose ramification at $t$, at least
over imaginary quadratic fields.

\begin{lem}\label{teo:level-lowering} Let $F$ be an
imaginary quadratic field such that $t$ and $q$, the primes where
$\pi$ ramifies, are split. Then among the primes $\ell$ that are
split in $F$, there are only finitely many $\ell$ such that the residual representation $\overline{\rho}_\lambda(\pi)$ is
unramified at $t$ for some $\lambda$ above $\ell$. In particular, the residual mod $\lambda$
representation contains a transvection given by the image of $I_t$
for all but finitely many primes $\ell$ that are split in
$F$, for every $\lambda \mid \ell$.
\end{lem}

\begin{proof} We restrict our reasoning to primes $\ell$ that are split in
$F$, different from $q,t$, and greater than $B$. The last condition allows us to
apply Lemma~\ref{teo:strongirreducible} for the residual mod~$\lambda$ representation.

Observe that, since $t$ is split in $F$, for every $\ell \neq t$ and every
$\lambda \mid \ell$, the restriction $\rho_\lambda(\pi)|_{G_F}$
ramifies at the primes in $F$ above $t$, and of course, if we assume
that the mod $\lambda$ representation becomes unramified at $t$ then
the same happens to its restriction to $G_F$. Thus, we can work over
$G_F$ and this is the setting where the problem of losing
ramification can be attacked because the available level-lowering
results only work, to this day, over CM fields. By (quadratic) base
change due to Arthur and Clozel, the restriction to $G_F$ of our compatible system is also
attached to an automorphic representation, namely the base change $\pi_F$
of $\pi$, an automorphic representation of $\GL_n(\A_F)$. (Note that $\pi_F$ is cuspidal since $\rho_\lambda(\pi)|_{G_F}$ is absolutely irreducible and that $\pi_F$ paired with the trivial character is a polarized representation in the sense of \cite{BLGGT}. Obviously $\pi_F$ is also regular algebraic, so we can speak of a compatible system attached to $\pi_F$.)

Thus, we take a prime $\ell$ split in $F$, $\ell > B$, and we assume
that $\overline{\rho}_\lambda(\pi)$ (and, a fortiori, its restriction to
$G_F$) is unramified at $t$ (the primes in $F$ dividing $t$,
respectively). To ease notation, let us call $\overline{\rho}_{F,\lambda}:=\overline{\rho}_{\lambda}(\pi)\vert_{G_F}$. This residual representation has an automorphic lift,
the one given by the $\lambda$-adic representation attached to
$\pi_F$, and since this $\lambda$-adic lift is ramified at the primes
above $t$ while $\overline{\rho}_{F,\lambda}$ is not, we can apply a level-lowering
result in this situation. Namely, we want to apply Theorem~4.4.1
from \cite{BLGGT} taking $S$ to be the ramification set
of $\bar{\rho}_{F,\lambda}$, i.e., we will take $S = \{\ell, q\}$.
In particular, $t$ will not be in~$S$. Let us first discuss the idea informally before going into details. The theorem gives, for a residual representation that is known to be automorphic, the existence of another
automorphic lift with prescribed local types, under some conditions that are met in our situation. We insist on
the new automorphic lift to have the same local type at ramified
primes as the given one except at the places above $t$ where we are
assuming the residual representation to be unramified, since we want
this new automorphic form to be unramified at the places above $t$. This is the reason why, in
our situation, this theorem can be considered as a level-lowering
result.

Going back to Theorem~4.4.1 of \cite{BLGGT}, we need to check the two conditions there. The $\lambda$-adic lift that we have, the one attached to $\pi_F$, is potentially diagonalisable at the primes
dividing $\ell$ whenever $\ell$ is sufficiently large, because then we
are in a Fontaine-Laffaille situation (cf.\ Lemma~1.4.3~(2) of~\cite{BLGGT} for a precise statement of the criterion for potential diagonalisability). This verifies the first condition. The second condition is immediately satisfied by Lemma~\ref{teo:strongirreducible} and $\ell>B$. We choose one place dividing
$q$ and one place dividing $\ell$ in $F$ and
if we call $S'$ the set of these two primes, for any $v \in S'$ we
fix as $n$-dimensional $\lambda$-adic representation $\rho_v$ of
$G_{F_v}$ the one obtained from the $\lambda$-adic representation
attached to $\pi_F$, restricted to $G_{F_v}$. Having this fixed, we
can apply Theorem~4.4.1 in \cite{BLGGT} to conclude that there exists
another automorphic representation $\pi'_F$ of $\GL_n(\A_F)$  such that:

\begin{enumerate}[(1)]

\item\label{item:pi'F1} $\pi'_F$ is cuspidal, regular algebraic, and polarizable in the sense of \cite{BLGGT} (more precisely RAECSDC when paired with a suitable Hecke character; this is analogous to the RAESDC representation in Section~\ref{sec:RAESDC} except that conjugate self--duality replaces
self--duality there), in particular equipped with an associated compatible system
$(\rho_\mu(\pi'_F))_\mu$, indexed by the primes $\mu$ of some number field~$M'$,

\item\label{item:pi'F2} $\pi'_F$ is unramified outside $S$,

\item\label{item:congruence} there is a place $\lambda'$ of~$M'$ above~$\ell$ such that
$ \overline{\rho}_{\lambda'}(\pi'_F) \cong \bar{\rho}_{F,\lambda}$,

\item\label{item:connects} if $v \in S'$, then $\rho_{\lambda'}(\pi'_F)|_{G_{F_v}}$ connects to
$\rho_v$ (see \cite{BLGGT} for the definition of {\it connects}).
\end{enumerate}

On the one hand, by known properties of the connected
relation (cf.\ \cite{BLGGT}, Section 1.4), Condition (\ref{item:connects}) at $v \mid \ell$  implies that the Weil-Deligne representations of
$\rho_{\lambda'}(\pi'_F)|_{G_{F_v}}$ and $\rho_v$, restricted to the inertia group at $v$, are isomorphic. Since $\rho_v$ is known to be crystalline, we can conclude that $\rho_{\lambda'}(\pi'_F)|_{G_{F_v}}$ is crystalline. Moreover the connected relation implies (by definition) that they have both the same Hodge-Tate numbers. From
this (and conjugate essential self-duality) we deduce that $\pi'_F$
has level prime to $\ell$ and that it has the same infinitesimal character as $\pi_F$.

On the other hand, at the prime in $S'$ above $q$, say $w$, the
Weil-Deligne representations corresponding to $\rho_{\lambda'}(\pi'_F)|_{G_{F_w}}$ and $\rho_w$ have isomorphic restrictions to
inertia. Indeed, since $\pi_F$ and $\pi'_F$ are cuspidal automorphic representations of $\GL(\A_F)$, the local components $\pi'_w$ and $\pi'_{F,w}$ are generic. Hence, by Lemma~1.3.2 of~\cite{BLGGT}, $\rho_{\lambda'}(\pi')|_{G_{F_w}}=\rho_w$ and $\rho_{\lambda'}(\pi'_F)|_{G_{F_w}}$ are smooth. From Lemma  1.3.4~(2) of~\cite{BLGGT}
(due to Choi) we conclude that the inertial types of $\pi_F$ and $\pi'_F$ at $w$ agree.

Observe that in particular, independently of $\ell$, the
automorphic representation $\pi'_F$ has fixed infinitesimal character at $\infty$,
fixed ramification set and fixed types at the ramified primes. It follows from the finiteness result of Harish--Chandra (cf.\ (1.7) and (4.4) of \cite{BorelJacquet}) that there are
only finitely many possibilities for $\pi'_F$. (We see from \cite{BorelJacquet} that there are finitely many $\pi'_F$ with fixed infinitesimal character at $\infty$ such that the finite part of $\pi'_F$ has a nonzero invariant vector under a fixed open compact subgroup of the finite part of $\GL_n(\A_{F})$. So it boils down to showing that the conductor of $\pi'_F$ is bounded. The latter can be seen via local Langlands from that we fix the ramification set of $\pi'_F$ as well as the types at the primes therein.)

Now assume that the
residual mod $\lambda$ representation attached to $\pi_F$ is unramified at $t$ for
\emph{infinitely} many primes $\ell$ (we keep the assumption that we are
only working with primes $\ell$ that are split in $F$). For each $\ell$ we find $\pi'_F$ as above (which a priori depends on $\ell$). Since there are only finitely many possibilities for $\pi'_F$ as $\ell$ varies, we conclude that there exists a $\pi'_F$ as in \eqref{item:pi'F1} and \eqref{item:pi'F2} such that the congruence (\ref{item:congruence}) and condition \eqref{item:connects} hold true for infinitely many $\lambda$.
But this has an immediate consequence
for the compatible systems attached to $\pi_F$ and $\pi'_F$. Since
there are congruences between these two systems in infinitely many
different residual characteristics, this forces the traces of both
systems at unramified places to be equal, and then using
Chebotarev's density theorem combined with the Brauer-Nesbitt theorem
we conclude that the two systems are isomorphic. But looking at the
restriction of these two systems at a decomposition group at $t$ we
get a contradiction, because due to compatibility with Local
Langlands, $(\rho_{\lambda}(\pi_F))_\lambda$ is known to be ramified at the places above $t$ while $(\rho_{\mu}(\pi'_F))_\mu$ is unramified at the places above~$t$ by construction.
This contradiction proves that the residual representation $\overline{\rho}_\lambda(\pi)$ can be
unramified at $t$ only for finitely many primes $\ell$ (among the primes that split in~$F$).
\end{proof}

We can apply the previous lemma over any quadratic number field of
the form $F= \Q(\sqrt{-w})$ where $w$ is a prime such that $q$ and
$t$ are split in~$F$. It is clear that there are infinitely many
such fields; consider an infinite sequence $(F_n)_n$ of distinct such fields. Let us call $T_n$ the set of primes that are split in one of the $F_i$, $i=1, \dots n$, and let $\calL$ be the set of primes $\ell$
such that $\overline{\rho}_\lambda(\pi)$ is unramified at $t$ for some
$\lambda \mid \ell$. Then $\calL\cap T_n$ is finite for any $n\in \mathbb{N}$. But as $n$ grows, the sets $T_n$ have (natural) density arbitrarily close to~$1$.
This clearly implies that $\calL$ must have density~$0$.

\section{Conclusion}\label{sec:conclusion}

The aim of this section is to show that the system of Galois representations studied in the previous two
sections, attached to the automorphic form constructed in Section~\ref{sec:existence},
does satisfy all the conditions in
the main result on the inverse Galois problem of~\cite{partII}.
Let us
check this in detail. (In the setting of \cite{partII},
we take $N_1=t$, $N_2=1$, $N=t$, and $L_0 = \Q$). First, observe that
the primes $q$ and $p$ as in Section~\ref{sec:compatible-systems} satisfy that $q$ is completely split in $\Q$,
and $p\equiv 1 \pmod n$, $p\mid q^n-1$ but $p\nmid q^{\frac{n}{2}}-1$.

The system $(\rho_{\lambda}(\pi))_\lambda$ of Galois representations of $G_{\Q}$ is a.~e.\ absolutely irreducible and
symplectic, and it satisfies the following properties:

\begin{itemize}
 \item The ramification set of the system is
$R=\{q, t\}$;
\item The system is Hodge-Tate regular with constant
Hodge-Tate weights and for every $\ell\not\in R$ and $\lambda\mid \ell$, the representation
$\rho_{\lambda}(\pi)$ is crystalline.
Let $a\in \mathbb{Z}$ be the smallest Hodge-Tate weight and let us call $k$ the biggest difference between any two Hodge-Tate numbers. By
Fontaine-Laffaille theory, we conclude that for every
$\ell\not\in R$, $\ell>k+2$, $\lambda\mid \ell$,
$\chi_{\ell}^{a}\otimes \overline{\rho}_{\lambda}(\pi)$ is regular in the sense of
\cite{partII}, and the tame inertia weights of this representation are bounded by~$k$ (in fact,
these weights for these $\ell$ agree with the Hodge-Tate numbers of
the system plus~$a$);

\item As we have seen in Section~\ref{sec:transvections}, for a density one set of primes $\ell\neq t$ and for every $\lambda\mid \ell$, the transvection corresponding to the image of $I_t$ under $\rho_{\lambda}(\pi)$ (cf.\ (\ref{item:att}) of Section~\ref{sec:compatible-systems})
is preserved in the reduction mod $\lambda$, hence $\overline{\rho}_{\lambda}(\pi)$ contains a nontrivial transvection;

\item As we have already observed (cf.\ (\ref{item:atq}) of Section~\ref{sec:compatible-systems}), for every $\ell\not=q$, for every
$\lambda\mid \ell$, the representation
$\rho_{\lambda}(\pi)|_{G_{\Q_q}}\cong \rho_q\otimes \alpha_{\lambda}$ for some unramified character $\alpha_{\lambda}:G_{\Q_q}\rightarrow \overline{M}_{\lambda}^{\times}$ and $\rho_q$ is, by
definition, $\mathrm{Ind}_{G_{\Q_{q^n}}}^{G_{\Q_q}}(\chi_q)$ for
$\chi_q$ as defined in Section~\ref{sec:existence};

\item For every $\ell\not=t$ and for every $\lambda\vert \ell$,
the image of $I_t$ under $\rho_{\lambda}(\pi)$ consists of either a group
generated by a transvection or the trivial group, so $\overline{\rho}_{\lambda}(\pi)(I_t)$ is in any case
an $\ell$-group, therefore it has order prime to $n!$ for any $\ell$
larger than $n$.
\end{itemize}
Thus, the main result on the inverse Galois problem of
\cite{partII} can be applied, and we deduce the following theorem:

\begin{thm}\label{teo:bigimage}
Let $\pi$ be the automorphic representation given by Theorem~\ref{t:existence-self-dual} of
Section~\ref{sec:existence}, with the ramified primes $q$ and $t$ satisfying the
conditions specified in Section~\ref{sec:compatible-systems}. Then, the compatible system
$(\overline{\rho}_{\lambda}(\pi))_\lambda$ has huge residual image for a density one
set of primes $\ell$, for every $\lambda \vert \ell$, i.e.,
$\mathrm{im}(\overline{\rho}_{\lambda}(\pi))$ is a huge subgroup of
$\GSp_n(\overline{\F}_\ell)$ for a density one set of primes.
\end{thm}

To derive our Galois theoretic result
(i.e.\ Theorem~\ref{teo:InverseGalois}), we just observe that, for any
given integer $d>0$, we can change the condition $p\equiv 1
\pmod{n}$ at the beginning of Section~\ref{sec:compatible-systems} by the stronger condition
$p\equiv 1\pmod{dn}$ while choosing $q$ and $t$ exactly as we did
in Section~\ref{sec:compatible-systems}.

With this, we conclude from the main result on the inverse Galois problem of
\cite{partII} that, for such an integer~$d$, the groups $\PGSp_n(\F_{\ell^d})$ or
$\PSp_n(\F_{\ell^d})$ are realised as Galois groups over~$\Q$ for a
positive density set of primes~$\ell$. Since this can be done for
any~$d$, Theorem~\ref{teo:InverseGalois} follows.

\bibliography{Bibliog}
\bibliographystyle{amsalpha}

\end{document}